\documentclass[intlim,righttag,10pt]{amsart}
\usepackage{amscd}
\usepackage{amssymb}
\usepackage[all]{xy}
\usepackage[french,british]{babel}
\newcommand{\Frob}{\mathrm{Frob}}
\newcommand{\kr}{\,\begin{picture}(-1,1)(-1,-2)\circle*{2}\end{picture}\ }

\DeclareMathOperator{\Spec}{\mathrm{Spec}}

\theoremstyle{plain}
\newtheorem{thm}{Theorem}
\newtheorem*{thm*}{Theorem}
\newtheorem{lem}[thm]{Lemma}

\newtheorem{prop}[thm]{Proposition}

\theoremstyle{remark}
\newtheorem{rem}[thm]{Remark}

\theoremstyle{definition}

\begin{document}

\title[Integral comparison]{Integral comparison of Monsky--Washnitzer and overconvergent de~Rham-Witt cohomology}

\author{Veronika Ertl and Johannes Sprang}
\address{Fakult\"at f\"ur Mathematik Universit\"at Regensburg  \\ 93053 Regensburg }
\email{}
\date{}
\date{\today}
\thanks{}
\maketitle 

\let\thefootnote\relax\footnote{The first author was supported by a habilitation grant through the Bavarian government. The second author was supported by DFG through CRC 1085. }

\begin{abstract}
{\noindent 
The goal of this small note is to extend a result by Christopher Davis and David Zureick-Brown  on the comparison between integral Monsky--Washnitzer cohomology and overconvergent de~Rham--Witt cohomology for a smooth variety over a perfect field of positive characteristic $p$ to all cohomological degrees independent of the dimension of the base or the prime number $p$. 
 }
\end{abstract}
\vspace{-9pt}
\selectlanguage{french}
\begin{abstract}
{\noindent
Le but de ce travail est de prolonger un r\'esultat de Christopher Davis et David Zureick-Brown concernant la comparaison entre la cohomologie de Monsky--Washnitzer enti\`ere et la cohomologie de de~Rham--Witt surconvergente d'une vari\'et\'e lisse sur un coprs parfait de charact\'eristique positive $p$  \`a tous les degr\'es cohomologiques ind\'epnedent de la dimension de base et du nombre premier $p$.
 }
\end{abstract}

\selectlanguage{british}

\section*{Introduction}

Let  $k$ be a perfect field of positive characteristic $p$. As usual denote by $W(k)$ the ring of $p$-typical Witt vectors of $k$, and  let $K$ be the fraction field of $W(k)$. By a variety over $k$ we always mean a separated and integral scheme of finite type over the field $k$.

In  \cite{davis_langer_zink} Christopher Davis, Andreas Langer and Thomas Zink define for a finitely generated $k$-algebra $\bar{A}$ the overconvergent de~Rham--Witt complex $W^\dagger\Omega^{\kr}_{\bar{A}/k}\subset W\Omega^{\kr}_{\bar{A}/k}$ which can be globalised to a sheaf on a smooth variety $X$ over $k$. One of their main results is to compare the cohomology of this complex to Monsky--Washnitzer cohomology. 

 According to Elkik (cf. \cite[Sec.~2]{vanderput}) there is a smooth finitely generated $W(k)$-algebra $A$ which reduces to $\bar{A}$.   Let $\widehat{A}$ be the $p$-adic completion of $A$. The weak completion $A^\dagger$ of $A$ is the smallest $p$-adically saturated subring of $\widehat{A}$ containing $A$ and all series of the form $\sum_{i_1,\ldots,i_n\geqslant 0}c_{i_1\ldots i_n}x_1^{i_1}\cdots x_n^{i_n}$ with $c_{i_1\ldots i_n}\in W(k)$ and $x_j\in pA^\dagger$. It is a weak formalisation of $\bar{A}$ in the sense of  \cite[Def. 3.2]{formal_cohomology} and according to \cite[(2.4.4)~Thm.]{vanderput} any two weak formalisations are isomorphic. By construction it is  weakly  finitely generated. For short it is common to write w.c.f.g for weakly complete and weakly finitely generated algebras.

Let $\widetilde{\Omega}^{\kr}_{A^\dagger/W(k)}$ be the module of continuous differentials of $A^\dagger$ over $W(k)$. The Monsky--Washnitzer cohomology of $\Spec \bar{A}$ is then calculated by the rational complex $\widetilde{\Omega}^{\kr}_{A^\dagger/W(k),\mathbb{Q}}:=\widetilde{\Omega}^{\kr}_{A^\dagger/W(k)}\otimes\mathbb{Q}$
	$$
	H_{\text{MW}}^i(\bar{A}/K)=H^i(\widetilde{\Omega}^{\kr}_{A^\dagger/W(k),\mathbb{Q}}).
	$$
These notions are well-defined and functorial \cite[Sec.~5]{formal_cohomology}.

A comparison map between the two complexes in question can be obtained as follows. For a smooth lift of the Frobenius $f:A^\dagger\rightarrow A^\dagger$, which always exists, there is a monomorphism
	$$
	t_f:A^\dagger\rightarrow W(\bar{A})
	$$
which has in fact image in the overconvergent subring $W^\dagger(\bar{A})\subset W(\bar{A})$ and induces by the universal property of K\"a{}hler differentials and functoriality a comparison map
	$$
	t_f:\widetilde{\Omega}^{\kr}_{A^\dagger/W(k)}\rightarrow W^\dagger\Omega^{\kr}_{\bar{A}/k}.
	$$
The main result of Christopher Davis, Andreas Langer, and Thomas Zink  regarding this comparison morphism is the following \cite[Cor.~3.25]{davis_langer_zink}.

\begin{thm*}[Davis--Langer--Zink]
\renewcommand{\labelenumi}{(\alph{enumi})} \begin{enumerate}\item Let $\dim \bar{A}<p$, then the map $t_{F}:\Omega^{\kr}_{A^\dagger/W(k)}\rightarrow W^\dagger\Omega^{\kr}_{\bar{A}/k}$ is a quasi-isomorphism.
\item In general, there is a rational isomorphism
	$$
	H_{\text{MW}}^\ast(\bar{A}/K)\cong H^\ast(W^\dagger\Omega^{\kr}_{\bar{A},\mathbb{Q}}),
	$$
between Monsky--Washnitzer cohomology and rational overconvergent de Rham--Witt cohomology.\end{enumerate}
\end{thm*}
Christopher Davis and David Zureick-Brown generalise the first statement which is an integral result to a comparison independent of the dimension of $\bar{A}$ \cite[Thm.~1.1]{davis_zureickbrown}.
\begin{thm*}[Davis--Zureick-Brown]
Let $\Spec\bar{A}$ be a non-singular affine variety and $A^\dagger$ a weak formalisation. 
\renewcommand{\labelenumi}{(\alph{enumi})} \begin{enumerate}
\item The integral Monsky--Washnitzer cohomology groups are well-defined.
\item For all $i<p$, we have an isomorphism $H^i(\widetilde{\Omega}^{\kr}_{A^\dagger/W(k)})\xrightarrow{\sim} H^i(W^\dagger\Omega^{\kr}_{\bar{A}/k})$.
\end{enumerate}
\end{thm*}

The goal of this paper is to take the argumentation of Davis and Zureick-Brown a step further and show that in the situation above the comparison map
	$$
	\widetilde{\Omega}^{\kr}_{A^\dagger/W(k)}\rightarrow W^\dagger\Omega^{\kr}_{\bar{A}/k}
	$$
induces a canonical isomorphism between integral cohomology groups
	$$
	H^i(\widetilde{\Omega}^{\kr}_{A^\dagger/W(k)})\rightarrow H^i(W^\dagger\Omega^{\kr}_{\bar{A}/k})
	$$
in all cohomological degrees. A crucial ingredient is the extension of a homotopy result by Monsky and Washnitzer. In \cite[Rem.~(3), p.~205]{formal_cohomology} they assert that if $A$ and $B$ are w.c.f.g. algebras over $W(k)$, with $A$  flat and $A/pA$ a complete transversal intersection, then two homomorphisms $\psi_1,\psi_2:A\rightarrow B$ with the same reduction modulo $p$ induce chain homotopic maps on the associated continuous de~Rham complexes
	$$\psi_1\cong \psi_2: \widetilde{\Omega}^{\kr}_{A/R}\rightarrow \widetilde{\Omega}^{\kr}_{B/R}.$$
In the first section, we show that this is in fact true without the assumptions that $A/pA$ is a complete transversal intersection and 	that $B$ is weakly finitely generated. We apply this in two instances.

In the second section we revisit integral Monsky--Washnitzer cohomology. Bearing  in mind \cite[Thm.~1.1(1)]{davis_zureickbrown}, it remains only to show that the cohomology groups $H^i_{\mathrm{MW}}(\Spec \bar{A}/W(k)):= H^i(\widetilde{\Omega}^{\kr}_{A^\dagger/W(k)})$ are functorial in $\bar{A}$.  Using the homotopy result of the previous section we don't have to reduce to transversal intersections, as the desired statement follows directly.

We turn our attention to the comparison map in the last section. As mentioned above, the problem is that it depends a priori on the choice of a Frobenius lift $f$. We use again the homotopy result mentioned above to show that two maps
$$\psi_1,\psi_2:A^\dagger\rightarrow W^\dagger(\bar{A})$$
which coincide on the reduction modulo $p$, induce chain homotopic maps on the associated complexes. At this instance it is important that the homotopy is valid in a case where the target is not weakly finitely generated.  It then suffices to invoke the universal property of the continuous de~Rham complex. Similar to \cite{davis_zureickbrown} we use now the fact that any smooth variety can be covered by special affines, on which we know that a quasi-isomorphism $\sigma: \widetilde{\Omega}^{\kr}_{A_I^\dagger/W(k)}\rightarrow W^\dagger\Omega_{\bar{A}_I/k}^{\kr}$ exists.

In summary we obtain the following result.

\begin{thm*} Let $\bar{A}$ be a smooth finite $k$-algebra.
\renewcommand{\labelenumi}{(\alph{enumi})} \begin{enumerate}
\item The integral Monsky--Wahsnitzer cohomology groups $H^i_{\mathrm{MW}}(\Spec \bar{A}/W(k))$ are well-defined up to unique isomorphism and are functorial in non-singular affine $k$-varieties.
\item For all $i\geqslant 0$ there is a well-defined and functorial isomorphism $H^i_{MW}(\bar{A}/W(k)) \xrightarrow{\simeq}  H^i(W^\dagger\Omega^{\kr}_{\bar{A}})$ between integral Monsky--Washnitzer and overconvergent de~Rham--Witt cohomology.
\end{enumerate}
\end{thm*}

\subsection*{Acknowledgements}
We are indebted to  Kennichi~Bannai and Kazuki~Yamada for helpful discussions and  suggestions related to the content of this paper.  Bernard~Le~Stum's insight on the topic, which was passed on to us by Christopher~J.~Davis, lead us to consider the paper of Alberto~Arabia. We would like to thank both of them for generously sharing their knowledge and giving us important feedback.

\section{A homotopy result}

 The heart of the following proposition is essentially  \cite[Rem.~(3), p.~205]{formal_cohomology} with the additional observation that it is neither necessary to assume that the target of the maps in question is weakly finitely generated nor that the reduction of the source is a complete transversal intersection. While the latter allows us to shorten the proofs in Sections 2 and 3 considerably, the first is crucial because we would like to apply the statement to maps $\psi_1,\psi_2 \colon A^\dagger \rightarrow W^\dagger({\bar{A}})$ for a smooth $k$-algebra $\bar{A}$ of finite type, and while $W^\dagger(\bar{A})$ is weakly complete \cite[Prop.~2.28]{davis_langer_zink2}, it is in general not weakly finitely generated. We recall the proof of \cite[Rem.~(3), p.~205]{formal_cohomology} with the necessary modifications.

\begin{prop}\label{Prop:homotopy}
Let $B$ be a weakly complete $W(k)$-algebra and $\bar{A}$ a smooth scheme of finite type over $k$. Let $A^\dagger$ be the weak completion of a smooth $W(k)$-lift $A$ of $\bar{A}$. Then two homomorphisms $\psi_1,\psi_2:A^\dagger\rightarrow B$ with the same reduction modulo $p$ induce chain homotopic maps on the associated continuous de~Rham complexes
	$$\psi_1\simeq \psi_2: \widetilde{\Omega}^{\kr}_{A^\dagger / W(k)}\rightarrow \widetilde{\Omega}^{\kr}_{B/W(k)}.$$
\end{prop}

\begin{proof}
A homotopy as desired can be obtained by introducing an extra variable $T$. Denote by $B\langle T\rangle$ the weak completion of the $W(k)[T]$-algebra $B[T]$ with respect to the ideal $(p,T)$.  The reduction of $B\langle T\rangle$ modulo the ideal $(p,T)$ is obviously $B/pB$ and according to \cite{formal_cohomology} we can think of it as restricted power series ring over $B$. Consider the natural maps 
	$$h_0,h_p:B\langle T\rangle\rightarrow B$$
sending $T$ to $0$ and $p$ respectively. Consider the complex of continuous differentials
	 $$\widetilde{\Omega}^{\kr}_{B\langle T\rangle /W(k)}= \Omega^{\kr}_{B\langle T\rangle /W(k)}/\bigcap  (p,T)^i\Omega^{\kr}_{B\langle T\rangle /W(k)},$$ 
 which is $(p,T)$-separated instead of $(p)$-separated. 

In a first step we observe that the induced maps 
	$$h_0,h_p: \widetilde{\Omega}^{\kr}_{B\langle T\rangle /W(k)}\rightarrow \widetilde{\Omega}^{\kr}_{B /W(k)}$$ 
are chain homotopic. Namely, by induction on the degree one shows easily that an element $\omega\in \widetilde{\Omega}^{\kr}_{B\langle T\rangle /W(k)}$ may be represented by a power series in $T$ as
\begin{equation}\label{omega}\omega=\sum_{i=0}^\infty T^i\left(\omega'_i+(dT\wedge\omega^{\prime\prime}_i)\right),\end{equation}
where for all $i\geqslant 0$ the elements $\omega'_i$ and $\omega^{\prime\prime}_i$ are in $\widetilde{\Omega}^{\kr}_{B /W(k)}$. For such a power series one sets
	$$L(\omega)=\sum_{i=0}^\infty\left(\frac{p^{i+1}}{i+1}\right)\omega^{\prime\prime}_i,$$
which is indeed a well-defined element of $\widetilde{\Omega}^{\kr}_{B /W(k)}$ because $i+1$ divides $p^{i+1}$ and moreover the fractions $\frac{p^{i+1}}{i+1}$ converge $p$-adically fast enough to zero. An easy computation using representations (\ref{omega})  shows that in each degree one obtains in fact an equality
	$$h_p-h_0=d_BL+Ld_{B\langle T\rangle}.$$

In a second step we show that two maps $\psi_1,\psi_2:A^\dagger\rightarrow B$ as in the statement of the proposition are strongly homotopic in the sense that there exists a map $\phi: A^\dagger\rightarrow B\langle T\rangle$ such that
$$h_0\circ \phi=\psi_1\qquad\text{ and }\qquad h_p\circ\phi =\psi_2.$$
 Let $C$ be the image of the map 
$$h_0\oplus h_p:B\langle T\rangle\rightarrow B\oplus B$$
which consists of pairs $(x,y)$ such that $\bar{x}=\bar{y}$ in $B/pB$. To make $h_0\oplus h_p$ a map of $W(k)[T]$-algebras one can give $C$ the structure of an $W(k)[T]$-algebra as which it is isomorphic to $B\langle T\rangle/\left(T(T-p)\right)$. Hence the reduction of $C$ modulo the ideal $(p,T)$ is $B/pB$ as well. As $\psi_1$ coincides with $\psi_2$ modulo $(p)$, the sum $\psi_1\oplus\psi_2:A^\dagger \rightarrow B\oplus B$ factors through $C$ and extends naturally to a map $A^\dagger\langle T\rangle\rightarrow C$. Modulo $(p,T)$ we obtain the diagram
$$\xymatrix{ && B/pB \ar[dd]^{\overline{h_0\oplus h_p}}\\&&\\ \bar{A}= A^\dagger/pA^\dagger \ar[uurr]^{\bar{\phi}} \ar[rr]^{\overline{\psi_1\oplus \psi_2}}& & B/pB}.$$
By \cite[Thm.~3.3.2~(b)]{arabia} the weak completion $A^\dagger\langle T\rangle$ of $\bar{A}$ over $(W(k)[T],(p,T))$ is very smooth and we make use of the relative lifting property \cite[Def.~2.4]{formal_cohomology} applied to the surjective map of weakly complete $W(k)[T]$-algbras $h_0\oplus h_p: B\langle T\rangle \rightarrow C$ in order to get a commutative diagram:
$$\xymatrix{ && B\langle T\rangle \ar[dd]^{h_0\oplus h_p}\\&& \\
A^\dagger\langle T\rangle \ar@{-->}[uurr]^{\exists \phi } \ar[rr]^{\psi_1\oplus \psi_2} && C}$$
Restricting $\phi$ to $A^\dagger$ results in the desired map.

Finally, putting the two observations together, we see that $L\circ \phi$ is a homotopy between $\psi_1$ and $\psi_2$.
\end{proof}

\section{Functoriality of integral Monsky--Washnitzer cohomology}

Let $\bar{A}$ be a smooth finite $k$-algebra. In \cite{davis_zureickbrown} Davis and Zureick-Brown prove the existence of an isomorphism
	$$	
	H^i(\widetilde{\Omega}^{\kr}_{A^\dagger/W(k)})\xrightarrow{\simeq} H^i(\widetilde{\Omega}^{\kr}_{(A')^\dagger/W(k)})
	$$
for two different smooth lifts $A$ and $A'$ with weak completions $A^\dagger$ and $(A')^\dagger$ of a non-singular affine $k$-variety $\bar{A}$. This section is nothing but the observation that their argument can also be used to prove functoriality of integral Monsky--Washnitzer cohomology in $\bar{A}$ in order to obtain the following result.

\begin{prop}\label{Prop:MWFunctoriality}
The cohomology groups
	$$	
	H^i_{\mathrm{MW}}(\Spec \bar{A}/W(k)):= H^i(\widetilde{\Omega}^{\kr}_{A^\dagger/W(k)})  
	$$
are well-defined up to unique isomorphism and are functorial in non-singular affine $k$-varieties.
\end{prop}

\begin{lem}\label{Lem:MapsOnCohCoin}
	Let $\Spec \bar{A}$ and $\Spec \bar{B}$ be two non-singular affine $k$-varieties and $\bar{\varphi}:\bar{A}\rightarrow\bar{B}$ a homomorphism. Let us choose two smooth lifts $A$ and $B$ over $W(k)$ with weak completions  $A^\dagger$ and $B^\dagger$ and two maps $\varphi_1,\varphi_2:A^\dagger\rightarrow B^\dagger$ lifting $\bar{\varphi}$. Then the induced maps
	$$	
	\varphi_1^*,\varphi_2^*\colon H^i(\widetilde{\Omega}^{\kr}_{A^\dagger/W(k)})\rightarrow H^i(\widetilde{\Omega}^{\kr}_{B^\dagger/W(k)})  
	$$
coincide.
\end{lem}

\begin{proof}
This is a special case of Proposition \ref{Prop:homotopy}
\end{proof}

\begin{proof}[Proof of Proposition \ref{Prop:MWFunctoriality}]
For the proof of the theorem it remains to show that the Monsky--Washnitzer cohomology for two different lifts are not only isomorphic but canonically isomorphic. For two different dagger algebras $A^\dagger$, $(A')^\dagger$ lifting $\bar{A}$ there is always a  lift
	$$	
	\varphi: A^\dagger\rightarrow (A')^\dagger 
	$$
of the identity which is in general not unique. By the independence of the lift on cohomology shown in the above lemma, there is a canonical isomorphism
	$$	
	H^i(\widetilde{\Omega}^{\kr}_{A^\dagger/W(k)})\xrightarrow{\simeq} H^i(\widetilde{\Omega}^{\kr}_{(A')^\dagger/W(k)}).
	$$
\end{proof}

\begin{rem}
Keeping in mind that for a smooth affine $k$-variety $X$ the rational Monsky--Washnitzer complex computes rigid cohomology our result identifies immediately a canonical $W(k)$-lattice on the cohomology groups $H_{\text{rig}}^i(X/K)$. What is more,  the functoriality of integral Monsky--Washnitzer cohomology induces such a lattice on cohomology groups for smooth quasi-projective $k$-schemes as well. Namely it allows us to  glue the integral structure along an appropriate finite cover of $X$ by smooth affine schemes to obtain the desired lattice on $H_{\text{rig}}^i(X/K)$.
\end{rem}

\section{An unconditional comparison}

In this section, we want to use intrinsic properties of weakly complete weakly finitely generated (w.c.f.g) algebras to obtain a comparison result between integral Monsky--Washnitzer and overconvergent de~Rham--Witt cohomology.

To define a comparison map we consider for a non-singular affine variety $\Spec\bar{A}$ over $k$, a weak formalisation $A^\dagger$  and a lifting $f:A^\dagger\rightarrow A^\dagger $ of the Frobenius morphism $\Frob:\bar{A}\rightarrow \bar{A}$.

Recursively, one can define a unique ring homomorphism 
	$$
	s_f:A^\dagger\rightarrow W(A^\dagger)
	$$
such that the ghost components of $s_f(a)$ for $a\in A^\dagger$ are given by $(a,f(a),f^2(a),\ldots)$. As noted in \cite[(0.1.3.16)]{illusie} it is functorial in the triple $(\bar{A},A^\dagger,f)$ in the sense that if $(\bar{A}',{A'}^\dagger, f')$ is another such triple and $\varphi:A^\dagger\rightarrow {A'}^\dagger$ a map commuting with the Frobenius lifts, i.e. the left square of the following diagram commutes, then the right diagram commutes as well
	$$
	\xymatrix{A^\dagger\ar[r]^f \ar[d]^\varphi & A^\dagger \ar[d]^\varphi \ar[r]^{s_f} & W(A^\dagger) \ar[d]^{W(\varphi)}\\ {A'}^\dagger \ar[r]^{f'} & {A'}^\dagger \ar[r]^{s_{f'}} & W({A'}^\dagger).}
	$$
 Let  $t_f=W(\pi)\circ s_f :A^\dagger\rightarrow W(\bar{A})$ be the composition of $s_f$ with the map induced by the reduction $\pi:A^\dagger\rightarrow \bar{A}$. According to \cite[Prop.~3.2]{davis_langer_zink} it factors through $W^\dagger(\bar{A})$ and one obtains
	$$
	t_f:A^\dagger\rightarrow W^\dagger(\bar{A})
	$$
which by the universal property of the continuous de~Rham complex  finally  results in the desired comparison map between complexes $t_f:\widetilde{\Omega}^{\kr}_{A^\dagger/W(k)}\rightarrow W^\dagger\Omega^{\kr}_{\bar{A}/k}$. One observes right away that the reduction of $t_f$ modulo $p$ is the identity. We aim to show, that the induced map on cohomology is an isomorphism which is independent of the choice of Frobenius lift.

\begin{lem}\label{Lem:Homotopy}
Let $\bar{A}$ be a smooth $k$-algebra of finite type and $A^\dagger$ a weak formalisation of $\bar{A}$ over $W(k)$. Let 
\[
	\psi_1,\psi_2 \colon A^\dagger \rightarrow W^\dagger({\bar{A}})
\]
be two morphisms which reduce to the same map modulo $p$. Then for every $i\geqslant 0$ the induced maps in cohomology
\[
	\psi_1,\psi_2: H^i(\widetilde{\Omega}^{\kr}_{A^\dagger/W(k)} )\rightarrow H^i(W^\dagger\Omega^{\kr}_{\bar{A}})
\]
are identical.
\end{lem}
\begin{proof}
	 Composing $\psi_i$, $i=1,2$ with the identity, we obtain by the universal property of the continuous de~Rham complex unique maps of differential graded algebras making the  diagram
	\begin{equation}\label{eq1}
		\xymatrix{\widetilde{\Omega}^{\kr}_{A^\dagger/W(k)} \ar@{-->}[r]^{\exists!} & \widetilde{\Omega}^{\kr}_{W^\dagger(\bar{A})/W(k)} \ar@{-->}[r]^{\exists!} & W^\dagger\Omega^{\kr}_{\bar{A}}\\
		A^\dagger\ar[r]^{\psi_i}\ar[u] & W^\dagger(\bar{A}) \ar[u]\ar[r]^{\text{id}} &  W^\dagger(\bar{A})\ar[u]
		}
	\end{equation}
	 commute. Let us call the map induced by $\psi_i$ on continuous de~Rham complexes $\tilde{\psi}_i: \widetilde{\Omega}^{\kr}_{A^\dagger/W(k)}\rightarrow  \widetilde{\Omega}^{\kr}_{W^\dagger(\bar{A})/W(k)} $. On the other hand, again by the universal property of the continuous de~Rham complex there is a unique map
\[
	\psi_i \colon \widetilde{\Omega}^{\kr}_{A^\dagger/W(k)} \rightarrow W^\dagger\Omega^{\kr}_{\bar{A}}
\]
	making
	\begin{equation}\label{eq2}
		\xymatrix{\widetilde{\Omega}^{\kr}_{A^\dagger/W(k)} \ar@{-->}[r]^{\exists!} & W^\dagger\Omega^{\kr}_{\bar{A}}\\
		A^\dagger\ar[r]^{\psi_i}\ar[u] & W^\dagger(\bar{A})\ar[u]
		}
	\end{equation}
	commute. But the upper composition in diagram \eqref{eq1} gives us another map making diagram \eqref{eq2} commute. By uniqueness they have to coincide. We can summarize the above discussion by saying that $\psi_1$ and $\psi_2$ factor as
	\begin{equation}\label{factorisation}
		\xymatrix{\widetilde{\Omega}^{\kr}_{A^\dagger/W(k)} \ar@<-.5ex>[r]_{\tilde{\psi}_2} \ar@<.5ex>[r]^{\tilde{\psi}_1} & \widetilde{\Omega}^{\kr}_{W^\dagger(\bar{A})/W(k)} \ar[r]^-{\text{can}} & W^\dagger\Omega^{\kr}_{\bar{A}}}.
	\end{equation}
	 Finally, keeping in mind that the reduction of $\tilde{\psi}_1$ and $\tilde{\psi}_2$ coincide  Proposition \ref{Prop:homotopy} shows that both maps  are homotopic. By the  factorisation (\ref{factorisation})  this implies that the same is true for $\psi_1$ and $\psi_2$.
\end{proof}

\begin{thm}
Let $\bar{A}$ be a non-singular affine variety over a perfect field of characteristic $p$. For all $i\geqslant 0$ there is a well-defined and functorial isomorphism
	\[
		H^i_{MW}(\bar{A}/W(k)) \stackrel{\simeq}{\rightarrow}  H^i(W^\dagger\Omega^{\kr}_{\bar{A}})
	\]
	between integral Monsky--Washnitzer and overconvergent de~Rham-Witt cohomology.
\end{thm}
\begin{proof}
	From here on, a similar proof as in \cite[Pf. of Thm. 1.1 (2)]{davis_zureickbrown} using a \v{C}ech spectral sequence argument applies. We recall it for completeness. Thus let $A^\dagger$ be a weak formalisation of $\bar{A}$ and $f:A^\dagger\rightarrow A^\dagger$ a lift of Frobenius.
	
Let $\mathcal{F}^{\kr}_{A^\dagger}$ be the sheaf of complexes associated to $\widetilde{\Omega}^{\kr}_{A^\dagger/W(k)}$ on $\Spec\bar{A}$. By \cite[Prop.~3.3]{davis_zureickbrown} the map $t_f$ from above induces  a morphism of complexes
	$$
	t_f: \mathcal{F}^{\kr}_{A^\dagger}\rightarrow W^\dagger\Omega_{\Spec{\bar{A}}/k}.
	$$
It is now possible to choose a cover $\mathcal{U}=\{U_i=\Spec\bar{A}_i\}$ of $\Spec{\bar{A}}$ by finitely many open special affines such that all finite intersections are of this form as well \cite[Prop.~3.5]{davis_zureickbrown}. ``Special'' in this context means the spectrum of  an algebra which is finite \'etale and monogenic  over the localisation of a polynomial algebra. For an arbitrary finite intersection of these opens $U_I=\Spec \bar{A}_I$ and a weak formalisation $A_I^\dagger$ we consider the induced map on cohomology
	$$
	t_f: \mathbb{H}^i(U_I,\mathcal{F}^{\kr}_{A^\dagger})=H^i(\widetilde{\Omega}^{\kr}_{A_I^\dagger/W(k)})\rightarrow H^i(W^\dagger\Omega_{\bar{A}_I/k}^{\kr})=\mathbb{H}^i(U_I,W^\dagger\Omega^{\kr}_{\Spec{\bar{A}/k}})$$
	where the first and last equality are due to the fact that the $\mathcal{F}^j_{A^\dagger}$ and $W^\dagger\Omega^j_{\Spec\bar{A}/k}$ have trivial sheaf cohomology for cohomological degree $i>0$. 
	
	By \cite[Thm.~3.19]{davis_langer_zink} there is a comparison  morphism $\sigma: \widetilde{\Omega}^{\kr}_{A_I^\dagger/W(k)}\rightarrow W^\dagger\Omega_{\bar{A}_I/k}^{\kr}$ for the special affine $\bar{A}_I$ which induces an isomorphism
	$$
	H^i(\widetilde{\Omega}^{\kr}_{A_I^\dagger/W(k)}) \xrightarrow{\sim} H^i(W^\dagger\Omega_{\bar{A}_I/k}^{\kr}).
	$$
Moreover, from the construction in \cite[(3.5)]{davis_langer_zink} it is immediately clear that $\sigma$ reduces to the identity modulo $p$, which as observed at the beginning of this section is also the case for $t_f$.  Applying Lemma \ref{Lem:Homotopy} we see that $t_f$ and $\sigma$ are homotopic. In particular, $t_f$ is a quasi-isomorphism on $\bar{A}_I$. 

For the induced morphism of \v{C}ech spectral sequences
	$$
	\xymatrix{ \check{H}^p(\mathcal{U}, \mathbb{H}^q(-,\mathcal{F}^{\kr}_{A^\dagger})) \ar[d]^{t_f}_{\cong}\ar@2{->}[r] & \mathbb{H}^{p+q}(\Spec \bar{A}, \mathcal{F}^{\kr}_{A^\dagger}) \ar[d]^{t_f}\\
	\check{H}^p(\mathcal{U}, \mathbb{H}^q(-,W^\dagger\Omega^{\kr}_{\bar{A}/k})) \ar@2{->}[r] &\mathbb{H}^{p+q}(\Spec \bar{A}, W^\dagger\Omega^{\kr}_{\bar{A}/k})}
	$$
this means  by \cite[Lem.~2.11]{davis_zureickbrown} that the fact that the morphisms on the left-hand side are isomorphisms show that the morphism on the right-hand side is one as well.
\end{proof}

\begin{rem}
It is worth to point out that the above comparison result indicates, as did already the result of Davis and Zureick-Brown for low cohomological degrees, that the cohomology groups of the  integral overconvergent de Rham--Witt complex are in general not finitely generated over $W(k)$. Monsky and Washnitzer mention in  \cite[Rem.~(3), p.~205]{formal_cohomology} as a counter example the affine line $\mathbb{A}_k^1=\Spec(k[T])$, for which the first cohomology group $H^1_{\text{MW}}(\mathbb{A}_k^1/W(k))$ is a huge torsion module. One can easily see this by considering the differentials $T^{p^n-1}dT$ in $\widetilde{\Omega}^{\kr}_{W(k)\langle T\rangle^\dagger/W(k)}$, which are closed but not exact.
\end{rem}

\end{document}